\documentclass[psamsfont,a4paper,12pt]{amsart}
\usepackage{pifont}
\usepackage[dvipsnames]{xcolor}
\usepackage{fullpage, amsmath, amssymb, amsthm, amscd, setspace, bm, graphicx, indentfirst, multirow, tikz, enumerate, verbatim, appendix}
\usepackage{adjustbox,amsfonts,array,graphicx,booktabs,tabularx,multirow,multicol,stmaryrd, tabu}
\usepackage[utf8]{inputenc}
\usepackage{cite}
\usepackage{hyperref}
\usepackage{mathabx,epsfig}
\usepackage{mathtools}
\usepackage{shuffle}
\usepackage{tikz-cd}
\usepackage{thmtools, thm-restate}

\date{\today}

\newtheorem{theorem}{Theorem}
\newtheorem{lemma}{Lemma}
\newtheorem{proposition}{Proposition}
\newtheorem{corollary}{Corollary}

\theoremstyle{definition}
\newtheorem{definition}{Definition}
\newtheorem{remark}{Remark}

\theoremstyle{remark}
\newtheorem{example}{Example}

\newcommand{\Z}{\mathbb{Z}}

\newcommand{\Q}{\mathbb{Q}}

\newcommand{\C}{\mathbb{C}}

\newcommand{\ba}{\overline}
\newcommand{\ul}{\underline}
\newcommand{\bs}{\boldsymbol}

\newcommand{\ol}{\overleftarrow}

\newcommand{\Sch}{\mathfrak{S}}
\newcommand{\Groth}{\mathcal{G}}
\newcommand{\RW}{\mathcal{RW}}
\newcommand{\HW}{\mathcal{HW}}

\newcommand{\code}{\text{code}}
\newcommand{\dualcode}{\text{dualcode}}
\DeclareMathOperator{\FS}{FS}
\DeclareMathOperator{\BS}{BS}

\DeclareMathOperator{\supp}{supp}
\renewcommand{\supp}{\mathop{\mathrm{supp}}\hspace{0.15em}}

\newcommand{\til}[1]{\widetilde{#1}}

\newcounter{Comment}

\title{Stability of products of double Grothendieck polynomials}

\author{Andrew Hardt}
\address[A.~Hardt]{Department of Mathematics, University of Illinois Urbana-Champaign, Urbana, IL 61801}
\email{ahardt@illinois.edu}
\urladdr{https://andyhardt.github.io/}

\author{David Wallach}
\address[D.~Wallach]{Department of Mathematics, University of Illinois Urbana-Champaign, Urbana, IL 61801}
\email{davidrw3@illinois.edu}

\begin{document}

\begin{abstract}
We prove that products of double Grothendieck polynomials have the same back- and forward-stability numbers as products of Schubert polynomials, characterize which simple reflections appear in such products, and also give a new proof of a finiteness conjecture of Lam-Lee-Shimozono on products of back-stable Grothendieck polynomials which was first proved by Anderson. To do this, we use the main theorems from our recent work, as well as expansion formulas of Lenart, Fomin-Kirillov, and Lam-Lee-Shimozono.
\end{abstract}

\maketitle

\section{Introduction}

This paper concerns products of \emph{double Grothendieck polynomials}, polynomial representatives $\Groth_w(\bs{x},\bs{y})$ in two sets of variables for the classes of Schubert varieties in the equivariant K-theory of the type A flag variety $GL_n/B$. First studied by Lascoux and Sch\"utzenberger \cite{LS-Grothendieck},
their lowest-degree terms, \emph{double Schubert polynomials} \cite{LascouxSchutzenberger}, are representatives for the equivariant cohomology classes of Schubert varieties.

Double Grothendieck polynomials for $GL_n/B$ are indexed by permutations in the symmetric group $S_n$ on $n$ letters. They satisfy the stability property $\Groth_w(\bs{x},\bs{y}) = \Groth_{w\times 1}(\bs{x},\bs{y})\in S_{n+1}$, where $w\times 1$ is the permutation acting by $w$ on $1,\ldots, n$ and fixing $n+1$. Thus, we can take $w$ to be an element of $S_{\Z_+}$, the set of bijections $\Z_+\to\Z_+$ fixing all but finitely-many elements. With this convention, double Grothendieck polynomials form a basis of the polynomial ring $\Q(\bs{y})[\bs{x}]$, and their structure constants $C_{u,v}^w(\bs{y})$, given by
\begin{equation} \label{eq:double-grothendieck-structure-constants}
\Groth_u(\bs{x},\bs{y})\Groth_v(\bs{x},\bs{y}) = \sum_w C_{u,v}^w(\bs{y}) \Groth_w(\bs{x},\bs{y}),
\end{equation}
satisfy a positivity property \cite{AGM-positivity} which generalizes that of double Schubert polynomials.

In the case where $u$ and $v$ are \emph{Grassmannian permutations} with the same descent, any $w$ appearing in \eqref{eq:double-grothendieck-structure-constants} must also be Grassmannian. These polynomials are representatives for equivariant K-classes of Schubert varieties considered as subvarieties of the Grassmannian. For single Grothendieck polynomials (the non-equivariant setting, where $\bs{y}$ is set to 0), Buch \cite{Buch-K-theory} proved a tableau rule for \eqref{eq:double-grothendieck-structure-constants} in the Grassmannian case.
An extension of Buch's formula to double Grothendieck polynomials was conjectured by Thomas and Yong \cite{ThomasYong}, and then proved by Pechenik and Yong \cite{PechenikYong1}.
Combinatorial formulas exist in a small number of other cases, such as 2-step flag varieties \cite{KnutsonZJ-I}.

However, for general $u$ and $v$, it remains a longstanding open problem to give a combinatorial description of $C_{u,v}^w(\bs{y})$, and relatively little is known even about which structure constants are nonzero. Our results here concern the general case, and we provide formulas for (i) the smallest $S_n$ containing all $w$ on the right side of \eqref{eq:double-grothendieck-structure-constants}, (ii) the smallest power of $\gamma$ (see next subsection) one must apply so that \eqref{eq:double-grothendieck-structure-constants} is $\gamma$-stable, and (iii) the set of all simple reflections appearing on the right side of \eqref{eq:double-grothendieck-structure-constants}.

A family of functions has \emph{finite product expansion} (FPE) if the product of any two of them is a finite linear combination of functions in the family. The vector space spanned by any such family has a ring structure. One notable feature of Buch's formulas \cite[Theorems~5.4, 6.13]{Buch-K-theory} is that they imply that K-Stanley symmetric functions, the stable limits of Grothendieck polynomials, have FPE.
Lam, Lee, and Shimozono \cite{LamLeeShimiozono-grothendieck} showed that FPE also holds for double K-Stanley symmetric functions. 
Those authors also explored the more difficult setting of the \emph{back-stable} limit. They proved that back-stable (single) Grothendieck polynomials have FPE, and conjectured the same for back-stable double Grothendieck polynomials. Anderson \cite{Anderson-strong-positivity} proved this conjecture using geometric methods, and we give a new proof in Corollary~\ref{cor:lls-conjecture}.

\subsection{Main Results}

A \emph{permutation} of a set $A$ is a bijection $A\to A$ fixing all but finitely-many elements. Let $S_\Z$ be the group of permutations of $\Z$, $S_{\Z_+}$ be the subgroup of permutations of $\Z_+ := \{1,2,\ldots\}$, and $S_n$ be the subgroup of permutations of $\{1,2,\ldots,n\}$.

We will utilize two sets of parameters, $\bs{x} = (x_1,x_2,\ldots)$ and $\bs{y} = (\ldots,y_{-1},y_0,y_1,\ldots)$\footnote{We need the parameters $y_i$ for $i\le 0$ since they, but not the analogous $x$ parameters, show up in the back-stable structure constants \eqref{eq:bs-struct-cnst-double-groth}.}.

\begin{definition} \label{def:double-groth} 
Let $w\in S_{\Z_+}$. The \emph{double Schubert polynomial} corresponding to $w$ is
\[
\Sch_w(\bs{x},\bs{y}) = \sum_{\substack{(t_1,\ldots,t_k)\in\RW(w) \\ 1\le b_1\le\cdots\le b_k \\ t_i\le t_{i+1}\implies b_i<b_{i+1} \\ b_i\le t_i}} \prod_{i=1}^k (x_{b_i} - y_{t_i-b_i+1}),
\]
the (ordinary) \emph{Schubert polynomial} is $\Sch_w(\bs{x}) := \Sch_w(\bs{x},\bs{0})$,
the \emph{double Grothendieck polynomial} corresponding to $w$ is 
\[
\Groth_w(\bs{x}, \bs{y}) = \sum_{\substack{(t_1,\ldots,t_k)\in\HW(w) \\ 1\le b_1\le\cdots\le b_k \\ t_i\le t_{i+1}\implies b_i<b_{i+1} \\ b_i\le t_i}} (-1)^{k-\ell(w)}\prod_{i=1}^k (x_{b_i} - y_{t_i-b_i+1}),
\]
and the (ordinary) \emph{Grothendieck polynomial} is $\Groth_w(\bs{x}) := \Groth_w(\bs{x},\bs{0})$.\footnote{Notice that the inequality $t_i\le t_{i+1}\implies b_i<b_{i+1}$ differs from that of \cite{HardtWallach}, where $t_i<t_{i+1}\implies b_i<b_{i+1}$ is used. The former convention is better for working with Hecke words, while the latter is more useful when working with slide polynomials.} Here $\RW(w)$ is the set of \emph{reduced words} for $w$, and $\HW(w)$ is the set of \emph{Hecke words} for $w$ (Section~\ref{sec:permutations}).
\end{definition}

Schubert and Grothendieck polynomials each form a $\Q$-basis for $\Q[\bs{x}]$, while double Schubert polynomials and double Grothendieck polynomials each form a $\Q(\bs{y})$-basis for $\Q(\bs{y})[\bs{x}]$. Their multiplication rules
\[
\Sch_u(\bs{x})\Sch_v(\bs{x}) = \sum_{w\in S_{\Z_+}} c_{u,v}^w\Sch_w(\bs{x}),
\qquad
\Sch_u(\bs{x},\bs{y})\Sch_v(\bs{x},\bs{y}) = \sum_{w\in S_{\Z_+}} c_{u,v}^w(\bs{y})\Sch_w(\bs{x},\bs{y}),
\]
\[
\Groth_u(\bs{x})\Groth_v(\bs{x}) = \sum_{w\in S_{\Z_+}} C_{u,v}^w\Groth_w(\bs{x}),
\qquad
\Groth_u(\bs{x},\bs{y})\Groth_v(\bs{x},\bs{y}) = \sum_{w\in S_{\Z_+}} C_{u,v}^w(\bs{y})\Groth_w(\bs{x},\bs{y}),
\]
are related under the formulas:
\begin{equation} \label{eq:struct-cnst-specialization}
C_{u,v}^w = C_{u,v}^w(\bs{0}), \qquad c_{u,v}^w(\bs{y}) = \begin{cases} C_{u,v}^w(\bs{y}), & \text{if } \ell(u) + \ell(v) = \ell(w), \\ 0, & \text{otherwise,}\end{cases}
\qquad c_{u,v}^w = c_{u,v}^w(\bs{0}).
\end{equation}
In addition, $c_{u,v}^w$ is always nonnegative, while $C_{u,v}^w$ has sign $(-1)^{\ell(w)-\ell(u)-\ell(v)}$ \cite{Brion-positivity}, and the double families satisfy similar positivity conditions \cite{Graham-positivity, AGM-positivity}.

We won't need back-stable double Grothendieck polynomials, but we will use their structure constants. Let $\gamma:S_\Z\to S_\Z$ be the map $\gamma(w)(i) = w(i+1)-1$, let $\gamma(y_i) = y_{i+1}$, and let $\gamma(\bs{y}) = (\ldots,\gamma(y_{-1}),\gamma(y_0),\gamma(y_1),\ldots)$.
If $w\in S_\Z$, let
\[
\FS(w) = \max\{k\mid w(k)\ne k\},\qquad\qquad \BS(w) = 1-\min\{k\mid w(k)\ne k\}.\footnote{This slightly differs from \cite{HardtWallach}, where $\BS(w)$ is required to be nonnegative, and instead matches $\til{BS}(w)$ from that paper. Similarly, $\BS(u,v)$ in this paper corresponds to $\til{BS}(u,v)$ from \cite{HardtWallach}.}
\]
For any $u,v,w\in S_\Z$, let 
\begin{equation} \label{eq:bs-struct-cnst-double-groth}
\ol{C_{u,v}^w}(\bs{y}) = C_{\gamma^k(u),\gamma^k(v)}^{\gamma^k(w)}(\gamma^{-k}(\bs{y})), \qquad\qquad \text{for any } k\ge \BS(u),\BS(v),\BS(w).
\end{equation}
By \cite[Proposition~8.25]{LamLeeShimiozono-grothendieck}, this is well-defined, and equals the structure constant for the corresponding product of back-stable double Grothendieck polynomials. Analogous to \eqref{eq:struct-cnst-specialization}, the specializations
\begin{equation} \label{eq:bs-struct-cnst-specialization}
\ol{C_{u,v}^w} := \ol{C_{u,v}^w}(\bs{0}), \quad \ol{c_{u,v}^w}(\bs{y}) := \begin{cases} \ol{C_{u,v}^w}(\bs{y}), & \text{if } \ell(u) + \ell(v) = \ell(w), \\ 0, & \text{otherwise,} \end{cases}
\qquad \ol{c_{u,v}^w} := \ol{c_{u,v}^w}(\bs{0})
\end{equation}
give the structure constants for back-stable Grothendieck polynomials, back-stable double Schubert polynomials, and back-stable Schubert polynomials, respectively \cite{LamLeeShimiozono-schubert,LamLeeShimiozono-grothendieck}.

If $u,v\in S_\Z$, the forward- and back-stability number
for the above products are
\begin{align*}
&\FS(u,v) = \max\{\FS(w)\mid c_{u,v}^w\ne 0\}, \qquad
&&\BS(u,v) = \max\{\BS(w)\mid \ol{c_{u,v}^w}\ne 0\},
\\&
\FS^D(u,v) = \max\{\FS(w)\mid c_{u,v}^w(\bs{y})\ne 0\}, \qquad
&&\BS^D(u,v) = \max\{\BS(w)\mid \ol{c_{u,v}^w}(\bs{y})\ne 0\},
\\&
\FS_K(u,v) = \max\{\FS(w)\mid C_{u,v}^w\ne 0\}, \qquad
&&\BS_K(u,v) = \max\{\BS(w)\mid \ol{C_{u,v}^w}\ne 0\},
\\&
\FS_K^D(u,v) = \max\{\FS(w) \mid C_{u,v}^w(\bs{y})\ne 0\}, \qquad
&&\BS_K^D(u,v) = \max\{\BS(w) \mid \ol{C_{u,v}^w}(\bs{y})\ne 0\}.
\end{align*}

Our main result is:

\begin{theorem} \label{thm:double-grothendieck-stability}
For all $u,v\in S_{\Z_+}$,
\begin{enumerate}
    \item[\textup{(a)}] $\FS(u,v) = \Xi(u,v)$ and $\BS(u,v) = \Omega(u,v)$,
    \item[\textup{(b)}] $\FS^D(u,v) = \Xi(u,v)$ and $\BS^D(u,v) = \Omega(u,v)$,
    \item[\textup{(c)}] $\FS_K(u,v) = \Xi(u,v)$ and $\BS_K(u,v) = \Omega(u,v)$,
    \item[\textup{(d)}] $\FS^D_K(u,v) = \Xi(u,v)$ and $\BS^D_K(u,v) = \Omega(u,v)$,
\end{enumerate}
where $\Xi(u,v), \Omega(u,v)$ are given by \eqref{eq:Xi-Omega-def}. 
\end{theorem}

Part (a) of Theorem~\ref{thm:double-grothendieck-stability} is \cite[Theorem~1.6, Proposition~6.3]{HardtWallach}, while the other parts are new. Part (d) implies the other parts; however, we will first use (a) to prove (c), and use that to prove (d) (and therefore (b) will follow). By \eqref{eq:struct-cnst-specialization}, we have
\begin{subequations} \label{eq:trivial-stability-inequalities}
\begin{align}
\FS(u,v)\le \FS_K(u,v)\le \FS^D_K(u,v), \qquad \FS(u,v)\le \FS^D(u,v)\le \FS^D_K(u,v),
\\
\BS(u,v)\le \BS_K(u,v)\le \BS^D_K(u,v), \qquad \BS(u,v)\le \BS^D(u,v)\le \BS^D_K(u,v),
\end{align}
\end{subequations}
so we only need to prove one direction of each equality.

\begin{remark}
The quantity $\FS^D_K(u,v)$ is the largest integer moved by any permutation showing up in the double Grothendieck product, while $1-\BS^D_K(u,v)$ is the smallest integer moved by any permutation showing up in the back-stable double Grothendieck product (and similarly for the other cases).

One of these definitions uses back-stable structure constants while the other does not. To understand the behavior in the other cases, one can show that
\[
\max\{\BS(w) \mid C_{u,v}^w(\bs{y})\ne 0\} = \min(\BS_K^D(u,v),0),
\]
and
\[
\max\{\FS(w) \mid \ol{C_{u,v}^w}(\bs{y})\ne 0\} = \FS_K^D(u,v).
\]
The first equation follows from the definitions. The second is not at all obvious, but is a consequence of \cite[Lemma~6.7]{HardtWallach} combined with \eqref{eq:bs-struct-cnst-double-groth} and Theorem~\ref{thm:double-grothendieck-stability}.
\end{remark}

Lam, Lee, and Shimozono showed \cite[(8.24)]{LamLeeShimiozono-grothendieck} that a product of back-stable double Grothendieck polynomials is always a linear combination of back-stable double Grothendieck polynomials, with coefficients given by \eqref{eq:bs-struct-cnst-double-groth}. They conjectured that the expansion was finite. This conjecture was proved by Anderson \cite{Anderson-strong-positivity} as a consequence of a general theorem on equivariant positivity.

\begin{corollary}[{\!\!\cite[Conjecture~8.27]{LamLeeShimiozono-grothendieck}, \cite[Corollary of Theorem B]{Anderson-strong-positivity}}] \label{cor:lls-conjecture}
For all $u,v\in S_\Z$, there are only finitely many $w$ such that $\ol{C_{u,v}^w}(\bs{y})\ne 0$.
\end{corollary}

Our work yields a new proof of this result which relies on older positivity results \cite{Brion-positivity}, but is otherwise combinatorial. We give two related arguments.

\begin{proof}
This follows from Theorem~\ref{thm:double-grothendieck-stability}(d), since there are only finitely-many permutations $w$ with $\BS(w)\le \Omega(u,v)$ and $\FS(w)\le \Xi(u,v)$, and this number is unchanged if we apply $\gamma$ to $u$ and $v$.

Alternatively, one can bypass Theorem~~\ref{thm:double-grothendieck-stability} by noting that there are only finitely-many permutations $w$ which appear on the right side of \eqref{eq:C-uv-w-y-expansion}, and by \cite[Theorem~8.11]{LamLeeShimiozono-grothendieck}, this number stabilizes after repeated applications of $\gamma$.\qedhere
\end{proof}

We also prove a formula (Corollary~\ref{cor:integer-support}) for the set of simple reflections appearing on the right side of the double Grothendieck product. In the case of single Schubert polynomials and dominant or Grassmannian permutations, this reduces to well-known facts.

\subsection*{Acknowledgements}

The authors are very grateful to Alexander Yong for suggesting we try to extend our previous results to this setting, as well as pointing us towards the expansion formula in \cite{Lenart-noncommutative}. We also thank Thomas Lam for making us aware of \cite{Anderson-strong-positivity}, and Cristian Lenart and Anna Weigandt for helpful discussions. We thank the referees for helpful comments and suggestions. This work benefited from computations using \textsc{SageMath}.

Both authors were partially supported by NSF RTG grant DMS-1937241. D.~W. was supported through the ICLUE program at the University of Illinois Urbana-Champaign.

\section{Permutations} \label{sec:permutations}

Let $s_i$ be the permutation with $s(i)=i+1, s(i+1)=i$, and for all other $j\in\Z$, $s(j)=j$. The $s_i$ generate $S_\Z$ and satisfy
\begin{equation} \label{eq:coxeter-relations}
s_i^2 = \text{id}, \qquad s_is_{i+1}s_i = s_{i+1}s_is_{i+1}, \qquad s_is_j = s_js_i \text{ if } |i-j|>1.
\end{equation}
If $w = s_{i_1}\cdots s_{i_k}$, then $(i_1,\ldots,i_k)$ is a \emph{word} for $w$. If $k$ is minimal among all such words, then $(i_1,\ldots,i_k)$ is a \emph{reduced word}, and the \emph{(Coxeter) length} of $w$ is $\ell(w) = k$. Let $\RW(w)$ be the set of all reduced words for $w$. Let $w_0^{(n)}$ be the unique permutation in $S_n$ of maximal length.

We also define the
\emph{0-Hecke product} (or \emph{Demazure product}):
\begin{equation} \label{eq:hecke-relations}
s_i\ast s_i = s_i, \qquad s_i\ast s_{i+1}\ast s_i = s_{i+1}\ast s_i\ast s_{i+1}, \qquad s_i\ast s_j = s_j\ast s_i \text{ if } |i-j|>1.
\end{equation}
If $w = s_{i_1}\ast\cdots \ast s_{i_k}$, then $(i_1,\ldots,i_k)$ is a \emph{Hecke word} for $w$. Let $\HW(w)$ be the set of all Hecke words for $w$.

The \emph{(strong) Bruhat order} $\le$ is the partial order on $S_\Z$ such that $u\le v$ if for some (equivalently, for every) reduced word $(i_1,\ldots,i_\ell)\in\RW(v)$, $u$ has a reduced word of the form $(i_{j_1},i_{j_2},\ldots,i_{j_k})$ where $j_1<j_2<\cdots<j_k$. The \emph{(weak) left Bruhat order} $\le_L$ is the partial order on $S_\Z$ such that $u\le_L v$ if $\ell(v) = \ell(vu^{-1}) + \ell(u)$, and the \emph{(weak) right Bruhat order} $\le_R$ is the partial order such that $u\le_L v$ if $\ell(v) = \ell(u^{-1}v) + \ell(u)$.
The \emph{one-line notation} of $w\in S_n$ is the sequence $w(1),\ldots,w(n)$, and we omit the commas as is standard practice.

\begin{definition} \label{def:lehmer-code}
The \emph{Lehmer code} is the doubly-infinite sequence
\[
\code(w) = (\ldots, c_{-1}, c_0, c_1, \ldots),
\]
where
\[
c_i := \code(w)_i = |\{j\in \Z \mid j>i, w(j)<w(i)\}|.
\]
The \emph{dual Lehmer code} is the doubly-infinite sequence
\[
\text{dualcode}(w) = (\ldots, d_{-1}, d_0, d_1, \ldots),
\]
where
\[
d_i := \text{dualcode}(w)_i = |\{j\in \Z \mid j<i, w(j)>w(i)\}|.
\]
Let
\[
\theta_i(w) = \begin{cases} 1, & \text{if code}(w)_i > 0,\\ 0, & \text{otherwise},\end{cases} \qquad\qquad \lambda_i(w) = \sum_{j\le i} \theta_j(w),
\]
and
\[\Theta_i(w) = \begin{cases}
    1 & \text{if dualcode}(w)_i>0, \\
    0 & \text{otherwise},
\end{cases}
\qquad\qquad
\Lambda_i(w) = \sum_{j\ge i} \Theta_j(w).
\]
\end{definition}

Let
\begin{subequations}  \label{eq:Xi-Omega-def}
\begin{align}
\Xi(u,v) = \max_{i\le 1+\max(\FS(u),\FS(v))}(\Lambda_i(u)+\Lambda_i(v)+i-1),
\\
\Omega(u,v) = \max_{i \geq -\max(\BS(u),\BS(v))}(\lambda_i(u) + \lambda_i(v) - i),
\end{align}
\end{subequations}
and by \cite[Theorem~1.6, Proposition~6.3]{HardtWallach}, $\FS(u,v) = \Xi(u,v)$ and $\BS(u,v) = \Omega(u,v)$.

\begin{example}
Let $u = 2147653$ and $v = 1547236$. In Figure~\ref{fig:main-example}, we compute $\Omega(u,v)=2$ and $\Xi(u,v)=10$ using the \emph{Rothe} and \emph{dual Rothe} diagrams of $u$ and $v$. For a permutation $w$, these diagrams are constructed by placing a dot in row $i$ of column $w(i)$, and extending lines to the bottom and right for the Rothe diagram, or top and left for the dual Rothe diagram. The number of boxes in row $i$ of the Rothe (resp. dual Rothe diagram) is $\code_i(w)$ (resp. $\dualcode_i(w)$), so the nonempty rows correspond to values of $i$ where $\theta_i(w)$ (resp. $\Theta_i(w)$) is nonzero.

$\lambda_i(w)$ can be computed by counting the nonempty rows in the Rothe diagram for $w$ from top to bottom, while $\Lambda_i(w)$ can be computed by counting the nonempty rows in the dual Rothe diagram from bottom to top. We write these numbers for each $i$ for both $u$ and $v$ to the left of the diagrams. The right column in the upper (resp. lower) portion of the figure lists the values of $\lambda_i(u)+\lambda_i(v)-i$ (resp. $\Lambda_i(u)+\Lambda_i(v)+i-1$) for each $i$, and we highlight the maximum value(s), which equals $\Omega(u,v)$ (resp. $\Xi(u,v)$).

\begin{figure}[h]
\begin{center}
\scalebox{.6}{$
\begin{array}{c@{\hspace{50pt}}c}
\begin{tikzpicture}
  \path[fill=black] (2,-1) circle (.1);
  \path[fill=black] (1,-2) circle (.1);
  \path[fill=black] (4,-3) circle (.1);
  \path[fill=black] (7,-4) circle (.1);
  \path[fill=black] (6,-5) circle (.1);
  \path[fill=black] (5,-6) circle (.1);
  \path[fill=black] (3,-7) circle (.1);
  \draw (2,-8)--(2,-1)--(8,-1);
  \draw (1,-8)--(1,-2)--(8,-2);
  \draw (4,-8)--(4,-3)--(8,-3);
  \draw (7,-8)--(7,-4)--(8,-4);
  \draw (6,-8)--(6,-5)--(8,-5);
  \draw (5,-8)--(5,-6)--(8,-6);
  \draw (3,-8)--(3,-7)--(8,-7);
  \draw[draw=black,line width=0.5mm] (0.5,-0.5) rectangle ++(1,-1);
  \draw[draw=black,line width=0.5mm] (2.5,-2.5) rectangle ++(1,-1);
  \draw[draw=black,line width=0.5mm] (2.5,-3.5) rectangle ++(1,-1);
  \draw[draw=black,line width=0.5mm] (4.5,-3.5) rectangle ++(1,-1);
  \draw[draw=black,line width=0.5mm] (5.5,-3.5) rectangle ++(1,-1);
  \draw[draw=black,line width=0.5mm] (2.5,-4.5) rectangle ++(1,-1);
  \draw[draw=black,line width=0.5mm] (4.5,-4.5) rectangle ++(1,-1);
  \draw[draw=black,line width=0.5mm] (2.5,-5.5) rectangle ++(1,-1);
  \node at (-0.5,0) {{\Large $\lambda_i(u)$}};
  \node at (-0.5,-1) {{\Large $1$}};
  \node at (-0.5,-2) {{\Large $1$}};
  \node at (-0.5,-3) {{\Large $2$}};
  \node at (-0.5,-4) {{\Large $3$}};
  \node at (-0.5,-5) {{\Large $4$}};
  \node at (-0.5,-6) {{\Large $5$}};
  \node at (-0.5,-7) {{\Large $5$}};
  \end{tikzpicture}
  &
\begin{tikzpicture}
  \path[fill=black] (1,-1) circle (.1);
  \path[fill=black] (5,-2) circle (.1);
  \path[fill=black] (4,-3) circle (.1);
  \path[fill=black] (7,-4) circle (.1);
  \path[fill=black] (2,-5) circle (.1);
  \path[fill=black] (3,-6) circle (.1);
  \path[fill=black] (6,-7) circle (.1);
  \draw (1,-8)--(1,-1)--(8,-1);
  \draw (5,-8)--(5,-2)--(8,-2);
  \draw (4,-8)--(4,-3)--(8,-3);
  \draw (7,-8)--(7,-4)--(8,-4);
  \draw (2,-8)--(2,-5)--(8,-5);
  \draw (3,-8)--(3,-6)--(8,-6);
  \draw (6,-8)--(6,-7)--(8,-7);
  \draw[draw=black,line width=0.5mm] (1.5,-1.5) rectangle ++(1,-1);
  \draw[draw=black,line width=0.5mm] (1.5,-2.5) rectangle ++(1,-1);
  \draw[draw=black,line width=0.5mm] (1.5,-3.5) rectangle ++(1,-1);
  \draw[draw=black,line width=0.5mm] (2.5,-1.5) rectangle ++(1,-1);
  \draw[draw=black,line width=0.5mm] (2.5,-2.5) rectangle ++(1,-1);
  \draw[draw=black,line width=0.5mm] (2.5,-3.5) rectangle ++(1,-1);
  \draw[draw=black,line width=0.5mm] (3.5,-1.5) rectangle ++(1,-1);
  \draw[draw=black,line width=0.5mm] (5.5,-3.5) rectangle ++(1,-1);
  \node at (-0.5,0) {{\Large $\lambda_i(v)$}};
  \node at (-0.5,-1) {{\Large $0$}};
  \node at (-0.5,-2) {{\Large $1$}};
  \node at (-0.5,-3) {{\Large $2$}};
  \node at (-0.5,-4) {{\Large $3$}};
  \node at (-0.5,-5) {{\Large $3$}};
  \node at (-0.5,-6) {{\Large $3$}};
  \node at (-0.5,-7) {{\Large $3$}};
  
  \node at (10,0) {{\Large $-i$}};
  \node at (10,-1) {{\Large $-1$}};
  \node at (10,-2) {{\Large $-2$}};
  \node at (10,-3) {{\Large $-3$}};
  \node at (10,-4) {{\Large $-4$}};
  \node at (10,-5) {{\Large $-5$}};
  \node at (10,-6) {{\Large $-6$}};
  \node at (10,-7) {{\Large $-7$}};
  
  \node at (12.5,0) {{\Large Total}};
  \node at (12.5,-1) {{\Large $0$}};
  \node at (12.5,-2) {{\Large $0$}};
  \node at (12.5,-3) {{\Large $1$}};
  \node at (12.5,-4) {{\Large $2$}};
  \node at (12.5,-5) {{\Large $2$}};
  \node at (12.5,-6) {{\Large $2$}};
  \node at (12.5,-7) {{\Large $1$}};
  \fill[orange,opacity=0.3] (12.5,-4) circle (.4);
  \fill[orange,opacity=0.3] (12.5,-5) circle (.4);
  \fill[orange,opacity=0.3] (12.5,-6) circle (.4);
  \end{tikzpicture}
\\\\
  \begin{tikzpicture}
  \path[fill=black] (2,-1) circle (.1);
  \path[fill=black] (1,-2) circle (.1);
  \path[fill=black] (4,-3) circle (.1);
  \path[fill=black] (7,-4) circle (.1);
  \path[fill=black] (6,-5) circle (.1);
  \path[fill=black] (5,-6) circle (.1);
  \path[fill=black] (3,-7) circle (.1);
  \draw (2,0)--(2,-1)--(0,-1);
  \draw (1,0)--(1,-2)--(0,-2);
  \draw (4,0)--(4,-3)--(0,-3);
  \draw (7,0)--(7,-4)--(0,-4);
  \draw (6,0)--(6,-5)--(0,-5);
  \draw (5,0)--(5,-6)--(0,-6);
  \draw (3,0)--(3,-7)--(0,-7);
  \draw[draw=black,line width=0.5mm] (1.5,-1.5) rectangle ++(1,-1);
  \draw[draw=black,line width=0.5mm] (6.5,-4.5) rectangle ++(1,-1);
  \draw[draw=black,line width=0.5mm] (3.5,-6.5) rectangle ++(1,-1);
  \draw[draw=black,line width=0.5mm] (4.5,-6.5) rectangle ++(1,-1);
  \draw[draw=black,line width=0.5mm] (5.5,-5.5) rectangle ++(1,-1);
  \draw[draw=black,line width=0.5mm] (5.5,-6.5) rectangle ++(1,-1);
  \draw[draw=black,line width=0.5mm] (6.5,-5.5) rectangle ++(1,-1);
  \draw[draw=black,line width=0.5mm] (6.5,-6.5) rectangle ++(1,-1);
  \node at (-0.5,0) {{\Large $\Lambda_i(u)$}};
  \node at (-0.5,-1) {{\Large $4$}};
  \node at (-0.5,-2) {{\Large $4$}};
  \node at (-0.5,-3) {{\Large $3$}};
  \node at (-0.5,-4) {{\Large $3$}};
  \node at (-0.5,-5) {{\Large $3$}};
  \node at (-0.5,-6) {{\Large $2$}};
  \node at (-0.5,-7) {{\Large $1$}};
  \end{tikzpicture}
  &
  \begin{tikzpicture}
  \path[fill=black] (1,-1) circle (.1);
  \path[fill=black] (5,-2) circle (.1);
  \path[fill=black] (4,-3) circle (.1);
  \path[fill=black] (7,-4) circle (.1);
  \path[fill=black] (2,-5) circle (.1);
  \path[fill=black] (3,-6) circle (.1);
  \path[fill=black] (6,-7) circle (.1);
  \draw (1,0)--(1,-1)--(0,-1);
  \draw (5,0)--(5,-2)--(0,-2);
  \draw (4,0)--(4,-3)--(0,-3);
  \draw (7,0)--(7,-4)--(0,-4);
  \draw (2,0)--(2,-5)--(0,-5);
  \draw (3,0)--(3,-6)--(0,-6);
  \draw (6,0)--(6,-7)--(0,-7);
  \draw[draw=black,line width=0.5mm] (3.5,-4.5) rectangle ++(1,-1);
  \draw[draw=black,line width=0.5mm] (3.5,-5.5) rectangle ++(1,-1);
  \draw[draw=black,line width=0.5mm] (4.5,-2.5) rectangle ++(1,-1);
  \draw[draw=black,line width=0.5mm] (4.5,-4.5) rectangle ++(1,-1);
  \draw[draw=black,line width=0.5mm] (4.5,-5.5) rectangle ++(1,-1);
  \draw[draw=black,line width=0.5mm] (6.5,-4.5) rectangle ++(1,-1);
  \draw[draw=black,line width=0.5mm] (6.5,-5.5) rectangle ++(1,-1);
  \draw[draw=black,line width=0.5mm] (6.5,-6.5) rectangle ++(1,-1);
  \node at (-0.5,0) {{\Large $\Lambda_i(v)$}};
  \node at (-0.5,-1) {{\Large $4$}};
  \node at (-0.5,-2) {{\Large $4$}};
  \node at (-0.5,-3) {{\Large $4$}};
  \node at (-0.5,-4) {{\Large $3$}};
  \node at (-0.5,-5) {{\Large $3$}};
  \node at (-0.5,-6) {{\Large $2$}};
  \node at (-0.5,-7) {{\Large $1$}};
  
  \node at (10,0) {{\Large $i-1$}};
  \node at (10,-1) {{\Large $0$}};
  \node at (10,-2) {{\Large $1$}};
  \node at (10,-3) {{\Large $2$}};
  \node at (10,-4) {{\Large $3$}};
  \node at (10,-5) {{\Large $4$}};
  \node at (10,-6) {{\Large $5$}};
  \node at (10,-7) {{\Large $6$}};
  
  \node at (12.5,0) {{\Large Total}};
  \node at (12.5,-1) {{\Large $8$}};
  \node at (12.5,-2) {{\Large $9$}};
  \node at (12.5,-3) {{\Large $9$}};
  \node at (12.5,-4) {{\Large $9$}};
  \node at (12.5,-5) {{\Large $10$}};
  \node at (12.5,-6) {{\Large $9$}};
  \node at (12.5,-7) {{\Large $8$}};
  \fill[orange,opacity=0.3] (12.5,-5) circle (.45);
  \end{tikzpicture}
  \end{array}$}
\end{center}
\caption{Top: Rothe diagrams for the permutations $u = 2147653$ and $v = 1547236$, along with the computation $\Omega(u,v) = 2$. Bottom: dual Rothe diagrams for $u$ and $v$, along with the computation $\Xi(u,v) = 10$.}
\label{fig:main-example}
\end{figure}

\end{example}

\section{Proof of the main theorem} \label{sec:double-grothendieck}

\subsection{Grothendieck stability} \label{sec:groth-stability}

Grothendieck polynomials can be expanded in the Schubert basis,
\[
\Groth_w(\bs{x}) = \sum_{x\in S_{\Z_+}} a_{x,w} \Sch_x(\bs{x}),
\]
for some constants $a_{x,w}$. We have $a_{w,w}=1$ for all $w$, and $a_{x,w}$ has sign $(-1)^{\ell(x)-\ell(w)}$ \cite{Lenart-noncommutative}.

The forward-stability part of Theorem~\ref{thm:double-grothendieck-stability}(c) is a consequence of the following lemma:

\begin{lemma} \label{lem:grothendieck-expansion-dual-lehmer}
Suppose that $a_{x,w}\ne 0$. Then $\Theta_i(x) \le \Theta_i(w)$ for all $i$.
\end{lemma}

In particular, $\FS(x)\le \FS(w)$ (in fact, they are equal). This lemma may be a surprising result, because whenever $a_{x,w}\ne 0$, $w\le x$ \cite[Proposition~5.8]{Lenart-noncommutative}.

\begin{proof}[Proof of Theorem~\ref{thm:double-grothendieck-stability}(c), forward stability part]
By \eqref{eq:trivial-stability-inequalities}, $\FS_K(u,v)\ge \FS(u,v)$. The other direction is proved by considering the Schubert expansion of the Grothendieck product:
\[
\Groth_u(\bs{x})\Groth_v(\bs{x}) = \sum_{x,y\in S_{\Z_+}} a_{x,u}a_{y,v} \Sch_x(\bs{x}) \Sch_y(\bs{x}) = \sum_{x,y\in S_{\Z_+}} a_{x,u}a_{y,v} \sum_{w\in S_{\Z_+}} c_{x,y}^w \Sch_w(\bs{x}),
\]
so expanding the left side and matching coefficients, we have
\[
\sum_{x\in S_{\Z_+}} C_{u,v}^x a_{w,x} = \sum_{x,y\in S_{\Z_+}} a_{x,u}a_{y,v} c_{x,y}^w, \qquad\qquad \text{for all } u,v,w\in S_{\Z_+}.
\]
For fixed $u,v,w$, every term on both sides has sign $(-1)^{\ell(w)-\ell(u)-\ell(v)}$, so there is no cancellation. Let $n = \max_w \FS(w)$ over all $w$ such that the above sum is nonzero. On one hand,
\[
n = \max\{\FS(w) \mid \exists\; x \text{ such that } C_{u,v}^x a_{w,x} \ne 0\} \ge \max\{\FS(w) \mid C_{u,v}^w \ne 0\} = \FS(u,v).
\]

On the other hand, by Lemma~\ref{lem:grothendieck-expansion-dual-lehmer}, if $a_{x,u}\ne 0$, then $\Theta_i(x)\le \Theta_i(u)$, for all $i$, so $\Lambda_i(x)\le \Lambda_i(u)$ for all $i$, and similarly, $\Lambda_i(y)\le \Lambda_i(v)$ for all $i$. Applying Theorem~\ref{thm:double-grothendieck-stability}(a), we have $\FS(x,y)\le \FS(u,v)$ whenever $a_{x,u},a_{y,v}\ne 0$. Therefore,
\[
n = \max_{x\mid a_{x,u}\ne 0} \max_{y\mid a_{y,v}\ne 0} \FS(x,y)
\le \FS(u,v).\qedhere
\]
\end{proof}

The proof of Lemma~\ref{lem:grothendieck-expansion-dual-lehmer} will use a formula by Lenart \cite{Lenart-noncommutative} for the $a_{x,w}$. Fix $n$. A \emph{binary tableau} (of staircase shape) is an array
\[
T = (t_{ij})_{1\le i\le j < n}, \qquad t_{ij}\in \{0,1\} \text{ for all } i,j.
\]
Let $u_1,\ldots,u_{n-1},v_1,\ldots,v_{n-1}$ be two sets of generators, which satisfy the relations
\[
v_i^2 = 0, \qquad v_iv_{i+1}v_i = v_{i+1}v_iv_{i+1}, \qquad v_iv_j = v_jv_i \text{ if } |i-j|>1,
\]
\[
u_i^2 = u_i, \qquad u_iu_{i+1}u_i = u_{i+1}u_iu_{i+1}, \qquad u_iu_j = u_ju_i \text{ if } |i-j|>1.
\]
(In the case of $u$, these are the relations of the $0$-Hecke product \eqref{eq:hecke-relations}.)
For a binary tableau $T$, define
\[
u(T) = \prod_{j=n-1}^1 \prod_{i=1}^j u_{n-i}^{1-t_{ij}}, \qquad\qquad v(T) = \prod_{i=1}^{n-1} \prod_{j=n-1}^i v_j^{t_{ij}}.
\]

We associate $u(T)$ with a unique permutation in $S_n$ (which we also call $u(T)$) by choosing a minimal-length representative and then sending $u_i\mapsto s_i$. The analogous identification holds for $v(T)$ whenever it is nonzero.

\begin{theorem}[{\!\!\cite[Theorem~5.3]{Lenart-noncommutative}}]
The coefficient $a_{x,w}$ is nonzero if and only if there exists a binary tableau $T$ with $u(T) = w, v(T) = xw_0^{(n)}$.
\footnote{In a preprint version of this work we conjectured that $a_{x,w}$ is nonzero if and only if there exists a saturated chain in the Bruhat order from $w$ to $x$ such that every covering relation $u\leq v$ in the chain is of the form $v = u(a,b)$ such that there exists some $a<c<b$ where $u(c)>u(b)$. Anna Weigandt has informed the authors she has found counterexamples to sufficiency, the smallest of which is $w=1423, x=3412$. On the other hand, necessity is still open, and a careful code check shows that this direction holds through at least $n=7$.}
\end{theorem}

This result is the main ingredient in the proof of Lemma~\ref{lem:grothendieck-expansion-dual-lehmer}.

\begin{proof}[Proof of Lemma~\ref{lem:grothendieck-expansion-dual-lehmer}]
We prove that for all binary tableaux $T$ and all $i$, 
\begin{equation} \label{eq:theta-u-v-relationship}
    \Theta_i(v(T)w_0^{(n)}) \le \Theta_i(u(T)).
\end{equation}
We will use induction on the tableau size $n$, in a similar manner to \cite[Proposition~5.8]{Lenart-noncommutative}. For $n=1$, \eqref{eq:theta-u-v-relationship} holds via a simple check, so assume $n>1$, and that \eqref{eq:theta-u-v-relationship} holds for binary tableaux of size $n-1$.

Let $T' = (t_{ij})_{1\le i < j\le n-1}$ be the binary tableau formed from $T$ by deleting all entries of the form $t_{ii}$. For $T'$, we keep the row and column indices the same as for $T$; the row indices will be $1,\ldots,n-2$ and the column indices will be $2,\ldots,n-1$. Let $\ba{T}$ be the same binary tableau at $T'$, but where the column indices are shifted to be $1,\ldots,n-2$. Let $u=u(T), v=v(T)w_0^{(n)}, u'=u(T'), v'=v(T')w_0^{(n)}, \ba{u}=u(\ba{T}), \ba{v}=v(\ba{T})w_0^{(n-1)}$. We have
\[
u'(j) = \begin{cases} 1, & \text{if } j=1, \\ \ba{u}(j-1)+1, & \text{otherwise},\end{cases}
\qquad\qquad
v'(j) = \begin{cases} 1, & \text{if } j=n, \\ \ba{v}(j)+1, & \text{otherwise},\end{cases}
\]
and
\[
\Theta_j(u') = \begin{cases} 0, & \text{if } j=1, \\ \Theta_{j-1}(\ba{u}), & \text{otherwise},\end{cases}
\qquad\qquad
\Theta_j(v') = \begin{cases} 1, & \text{if } j=n, \\ \Theta_j(\ba{v}), & \text{otherwise}.\end{cases}
\]
See Example~\ref{ex:binary-tableaux} for an example of these computations.

Now, $\ba{T}$ is a binary tableau of size $n-1$, so by the inductive hypothesis, $\Theta_i(\ba{v}) \le \Theta_i(\ba{u})$ for all $1\le i\le n-1$. By \cite[(5.9)]{Lenart-noncommutative}, there exist subsets
\[
I = \{i_1<i_2<\cdots < i_k\} \subseteq \{1,\cdots,n-1\},
\]
\[
J = \{j_1<j_2<\cdots < j_{n-k-1}\} = \{1,\cdots,n-1\} \setminus I
\]
such that
$u = u'u_{i_1}\cdots u_{i_k}, v = v' v_{j_{n-k-1}}\cdots v_1$ and $\ell(u) = \ell(u') + k$, $\ell(v) = \ell(v') - (n-k-1)$. (A convenient way to think about these expressions is that $u_{i_1}\cdots u_{i_k}$ acts on the right of the one-line notation of $u'$ by dragging entries to the right in a way that increases the length of the permutation, while $v_{j_1}\cdots v_{j_{n-k-1}}$ acts on the right of the one-line notation of $v'$ by dragging entries to the left in a way that decreases the length of the permutation.)

Next, we show that $\Theta_a(v) \le \Theta_a(u)$ for all $a$. Since $\Theta_1(v) = 0 = \Theta_1(u)$, this holds whenever $a=1$. If $n-1\in I$, then $u(n-1)>u(n)$, so $\Theta_n(u)=1\ge \Theta_n(v)$. If $n-1\in J$, then $v'(n-1)>v'(n)$, so $\Theta_n(v')=0$, a contradiction, so $\Theta_a(v) \le \Theta_a(u)$ holds for $a=n$.

Now, suppose that $1<a<n$.

\emph{Case 1}. If $a-1\notin I,a\notin I$, then $\Theta_a(u) = \Theta_a(u') = \Theta_{a-1}(\ba{u})$. In addition, $a-1\in J,a\in J$, so $\Theta_a(v) = \Theta_{a-1}(v') = \Theta_{a-1}(\ba{v})$, where the first equality holds because $v(a) = v'(a-1)$ and the unique $b$ such that $v^{-1}(b) < a, (v')^{-1}(b) > a-1$ satisfies $b<v(a)$. Thus, we have $\Theta_a(v) = \Theta_{a-1}(\ba{v}) \le \Theta_{a-1}(\ba{u}) = \Theta_a(u)$.

\emph{Case 2}. If $a-1\in I, a\notin I$, then $u(a)<u(a-1)$, so $\Theta_a(u) = 1$, and we have $\Theta_a(v) \le 1 = \Theta_a(u)$.

\emph{Case 3}. If $a-1\notin I, a\in I$, then $\Theta_a(u) = \Theta_a(u')$ if $u(a)>u(a+1)$, and $\Theta_{a+1}(u')$ otherwise, so $\Theta_a(u) = \min(\Theta_a(u'), \Theta_{a+1}(u')) = \min(\Theta_{a-1}(\ba{u}), \Theta_a(\ba{u}))$. In addition, $a-1\in J, a\notin J$, so in particular $v'(a-1)>v'(a)$, so $\Theta_a(\ba{v}) = \Theta_a(v') = 1$. Since $\Theta_a(\ba{v})\le \Theta_a(\ba{u})$, we also have $\Theta_a(\ba{u}) = 1$, so $\Theta_a(u) = \Theta_{a-1}(\ba{u})$. Thus, we have
\[
\Theta_a(v) = \Theta_{a-1}(v') = \Theta_{a-1}(\ba{v})\le \Theta_{a-1}(\ba{u}) = \Theta_a(u),
\]
where the first equality is because $v(a) = v'(a-1)$ and any $b>v(a)$ satisfies $v^{-1}(b) < a \iff (v')^{-1}(b) < a-1$.

\emph{Case 4}. If $a-1\in I,a\in I$, then $\Theta_a(u) = \Theta_{a+1}(u') = \Theta_a(\ba{u})$, where the first equality follows from a simple case check, depending on the relative ordering of $u'(a-1),u'(a)$, and the (unique if it exists) $b$ such that $u^{-1}(b) > a, (u')^{-1}(b)<a-1$. In addition, $a-1\notin J,a\notin J$, so $\Theta_a(v) = \Theta_a(v') = \Theta_a(\ba{v})$. Thus, $\Theta_a(v) = \Theta_a(\ba{v})\le \Theta_a(\ba{u}) = \Theta_a(u)$.\qedhere
\end{proof}

\begin{example} \label{ex:binary-tableaux}
Let $T$ be the following binary tableau:
\[
T =\; \begin{matrix} t_{14}&t_{13}&t_{12}&t_{11} \\ t_{24}&t_{23}&t_{22} \\ t_{34}&t_{33} \\ t_{44}\end{matrix} \;=\; \begin{matrix} 0&1&0&1 \\ 1&1&0 \\ 0&0 \\ 0\end{matrix}.
\]
Then,
\[
T' =\; \begin{matrix} t_{14}&t_{13}&t_{12} \\ t_{24}&t_{23} \\ t_{34}\end{matrix} \;=\; \begin{matrix} 0&1&0 \\ 1&1 \\ 0\end{matrix}, \qquad \text{and} \qquad \ba{T} \; =\; \begin{matrix} \ba{t_{13}}&\ba{t_{12}}&\ba{t_{11}} \\ \ba{t_{23}}&\ba{t_{22}} \\ \ba{t_{33}}\end{matrix} \;=\; \begin{matrix} 0&1&0 \\ 1&1 \\ 0\end{matrix}.
\]
We have
\[
\ba{u} = u(\ba{T}) = u_3u_1u_3 = s_3s_1 = 2\ul{1}4\ul{3}, \qquad v(\ba{T}) = v_2v_3v_2 = s_2s_3s_2 = 1432,
\]
\[
u' = u(T') = u_4u_2u_4 = s_4s_2 = 13\ul{2}5\ul{4}, \qquad v(T') = v_3v_4v_3 = s_3s_4s_3 = 12543,
\]
\[
u = u(T) = u_4u_2u_1u_2u_4u_3 = s_4s_2s_1s_2s_3 = 3\ul{2}5\ul{1}\ul{4}, \quad v(T) = v_3v_1v_4v_3 = s_3s_1s_4s_3 = 21543,
\]
\[
\ba{v} = v(\ba{T})w_0^{(n-1)} = 234\ul{1}, \qquad v' = v(T')w_0^{(n)} = 345\ul{2}\ul{1}, \qquad v = v(T)w_0^{(n)} = 345\ul{1}\ul{2},
\]
where we have underlined the entries of the relevant permutations where $\Theta=1$.
Comparing $u'$ with $u$ and $v'$ with $v$, we see that $I = \{1,2,3\}, J = \{4\}$. The lemma is confirmed in this case: for every underlined entry of $v$, the corresponding entry of $u$ is also underlined.
\end{example}

\begin{proof}[Proof of Theorem~\ref{thm:double-grothendieck-stability}(c), back-stability part]
The Grothendieck structure constants satisfy the symmetry (see e.g. \cite{KnutsonYong})
\begin{equation} \label{eq:w0-conjugation}
C_{u,v}^w = C_{w_0uw_0,w_0vw_0}^{w_0ww_0} \qquad \text{for any } u,v,w\in S_n \text{ and } w_0:=w_0^{(n)}.
\end{equation}
We have $w_0ww_0(n+1-i) = n+1-w(i)$, and from this it follows that
\begin{equation} \label{eq:code-dualcode-relation}
\BS(w) = \FS(w_0ww_0)-n \qquad \text{and} \qquad
\code(w)_i = \dualcode(w_0ww_0)_{n+1-i}.
\end{equation}
Therefore,
\begin{align*}
\Omega(u,v) &= \max_{i \geq -\max(\BS(u),\BS(v)}(\lambda_i(u) + \lambda_i(v) - i)
\\&= \max_{i \geq n-\max(\FS(w_0uw_0),\FS(w_0vw_0))}(\Lambda_{n+1-i}(w_0uw_0) + \Lambda_{n+1-i}(w_0vw_0) - i)
\\&= \max_{i \leq 1+\max(\FS(w_0uw_0),\FS(w_0vw_0))}(\Lambda_i(w_0uw_0) + \Lambda_i(w_0vw_0) - n-1+i)
\\&= \Xi(w_0uw_0,w_0vw_0) - n,
\end{align*}
so by the forward-stability part of Theorem~\ref{thm:double-grothendieck-stability}(c),
\[
\FS_K(w_0uw_0,w_0vw_0) = \Xi(w_0uw_0,w_0vw_0) = \Omega(u,v)+n.
\]
Also note that $\BS(\gamma^k(w)) = \BS(w)-k$ and $\Omega(\gamma^k(u), \gamma^k(v)) = \Omega(u,v)-k$.

By \eqref{eq:trivial-stability-inequalities} and Theorem~\ref{thm:double-grothendieck-stability}(a), $\Omega(u,v) = \BS(u,v) \le \BS_K(u,v)$. To prove the other inequality, take any $w\in S_\Z$ such that $\ol{C_{u,v}^w}\ne 0$. Let $k\ge \BS(w)$, and let $u' = \gamma^k(u), v' = \gamma^k(v), w' = \gamma^k(w)$. By \eqref{eq:bs-struct-cnst-double-groth} and \eqref{eq:w0-conjugation},
\[
\ol{C_{u,v}^w} = C_{u',v'}^{w'} = C_{w_0u'w_0,w_0v'w_0}^{w_0w'w_0},
\]
where $w_0:=w_0^{(n)}$ for $n$ taken large enough so that $u',v',w'\in S_n$.
Therefore, 
\begin{align*}
\BS(w) + n - k &= \BS(w') + n \\
&= \FS(w_0w'w_0) \\ 
&\le \FS_K(w_0u'w_0,w_0v'w_0) \\
&= \Xi(w_0u'w_0,w_0v'w_0) \\
&= \Omega(u',v') + n \\
&= \Omega(u,v) + n - k,
\end{align*}

so $\BS_K(u,v) = \max(\{\BS(w) \mid \ol{C_{u,v}^w} \neq 0\}) \leq \Omega(u,v)$.
\qedhere
\end{proof}

\subsection{Double Grothendieck stability}

To prove part (d) of Theorem~\ref{thm:double-grothendieck-stability}, we will use a similar method to the previous subsection: expand each double Grothendieck polynomial as a sum of single Grothendieck polynomials, take the product, and re-expand in terms of double-Grothendieck polynomials. The key formula here is the Cauchy identity first proven by Fomin and Kirillov \cite{FominKirillov-grothendieck}, which can be written as an expansion of double Grothendieck polynomials into single Grothendieck polynomials, as in the proof of \cite[Theorem~6.7]{LenartRobinsonSottile}. Lam, Lee, and Shimozono \cite{LamLeeShimiozono-grothendieck} give the inverse expansion, and for consistency we will use their version of both formulas. By applying these expansions, we obtain a formula \eqref{eq:C-uv-w-y-expansion} for $C_{u,v}^w(\bs{y})$ as a $\Q(\bs{y})$-linear combination of the $C_{u,v}^w$. This formula has cancellation, but it is sufficient for our purposes.

Given $u,w$, there exists $v$ such that $v\ast u = w$ if and only if $u\le_L w$. Let $\ominus y = \frac{-y}{1-y}$, and let $\ominus\bs{y} = (\ldots, \ominus y_0, \ominus y_1,\ominus y_2,\ldots)$. 

\begin{proposition}[{\!\!\cite[Proposition~5.8]{LamLeeShimiozono-grothendieck}}] \label{prop:double-grothendieck-expansion} 
For all $w\in S_{\Z_+}$,
\[
\Groth_w(\bs{x},\bs{y}) = \sum_{u\le_L w}A_{w,u}(\bs{y})\Groth_u(\bs{x}),
\quad\text{where}\quad
A_{w,u}(\bs{y}) = \sum_{v\ast u = w} (-1)^{\ell(w)-\ell(u)-\ell(v)} \Groth_{v^{-1}}(\ominus\bs{y}),
\]
\[
\Groth_w(\bs{x}) = \sum_{u\le_L w}B_{w,u}(\bs{y})\Groth_u(\bs{x},\bs{y}),
\quad\text{where}\quad
B_{w,u}(\bs{y}) = \sum_{v\ast u = w} (-1)^{\ell(w)-\ell(u)-\ell(v)} \Groth_{v}(\bs{y}).
\]
\end{proposition}

Note that $A_{w,u}(\bs{y})$ and $B_{w,u}(\bs{y})$ are specializations of the \emph{biaxial Grothendieck polynomial} studied in \cite{BFHTW-frozen-pipes}. Since (single) Grothendieck polynomials are linearly independent, $A_{w,u}(\bs{y})$ and $B_{w,u}(\bs{y})$ are always nonzero when $u\le_L w$.

\begin{lemma} \label{lem:weak-bruhat-Lambda}
    If $u\le_L w$, then for all $j$, $\Theta_j(u)\le\Theta_j(w)$ and  $\theta_j(u)\le\theta_j(w)$.
\end{lemma}

\begin{proof}
We first show that $\Theta_j(u)\le\Theta_j(w)$ for all $j$, and use induction on the left Bruhat order. Suppose that $w$ covers $u$, i.e. $w = s_iu$ for some $i$ and $\ell(w) = \ell(v)+1$. Let $a = u^{-1}(i), b = u^{-1}(i+1)$. We prove that $\dualcode_b(w) = \dualcode_b(u)+1$, and that for all other $j$, $\dualcode_j(w) = \dualcode_j(u)$. This shows that $\dualcode_j(u) \le \dualcode_j(w)$ for all $j$, meaning that $\Theta_j(u)\le\Theta_j(w)$.

We have $w(a)=i+1, w(b)=i$, and for all other $j$, $w(j)=u(j)$. Since $u\le w$, we have $a<b$. For $j\ne b$,
we have $\dualcode_j(w) = \dualcode_j(u)$ since $w(j) < i\iff w(j) < i+1$. On the other hand, $\dualcode_b(w) = \dualcode_b(u)+1$ since $w(b) = i < i+1 = w(a)$ while $u(b) = i+1 > i = u(a)$, and for all $j<b$ with $j\ne a$, $w(b) < w(j) \iff u(b) < u(j)$.

Now we show that $\theta_j(v)\le\theta_j(w)$ for all $j$. Choose $n$ large enough so that $v,w\in S_n$. By
\eqref{eq:code-dualcode-relation}, $\theta_j(u) = \Theta_{n+1-j}(w_0uw_0)$, and similarly for $w$. The Bruhat inequality $u\le_L w$ implies that $w_0uw_0\le_L w_0ww_0$, so by the first part of this proof,
\[
\theta_j(u) = \Theta_{n+1-j}(w_0uw_0) \le \Theta_{n+1-j}(w_0ww_0) = \theta_j(w).\qedhere
\]
\end{proof}

\begin{proof}[Proof of Theorem~\ref{thm:double-grothendieck-stability}(d)]
By \eqref{eq:trivial-stability-inequalities}, $\FS_K^D(u,v)\ge \FS_K(u,v)$ and $\BS_K^D(u,v)\ge \BS_K(u,v)$.

The other direction is proved by considering the Grothendieck expansion of the double Grothendieck product. By Proposition~\ref{prop:double-grothendieck-expansion},
\begin{align*}
\Groth_u(\bs{x},\bs{y})\Groth_v(\bs{x},\bs{y})
&= \sum_{\substack{a\le_L u\\ b\le_L v}} \Groth_a(\bs{x}) A_{u,a}(\bs{y}) \Groth_b(\bs{x}) A_{v,b}(\bs{y})
\\&= \sum_z \Groth_z(\bs{x}) \left(\sum_{\substack{a\le_L u\\ b\le_L v}} C_{a,b}^z A_{u,a}(\bs{y}) A_{v,b}(\bs{y})\right).
\\&= \sum_{w} \sum_{z\ge_L w} B_{z,w}(\bs{y}) \left(\sum_{\substack{a\le_L u\\ b\le_L v}} C_{a,b}^z A_{u,a}(\bs{y}) A_{v,b}(\bs{y})\right)\Groth_w(\bs{x},\bs{y}),
\end{align*}
so picking off coefficients and expanding the right side,
\begin{equation} \label{eq:C-uv-w-y-expansion}
C_{u,v}^w(\bs{y})
= \sum_{\substack{a\le_L u\\ b\le_L v \\ z\ge_L w}} \sum_{\substack{c\ast a = u\\ d\ast b = v \\ e\ast w = z}} (-1)^{\ell(u)-\ell(a)-\ell(c)+\ell(v)-\ell(d)-\ell(b)+\ell(w)-\ell(e)-\ell(z)} C_{a,b}^z \Groth_{c^{-1}}(\ominus\bs{y}) \Groth_{d^{-1}}(\ominus\bs{y}) \Groth_e(\bs{y}).
\end{equation}

In particular, $C_{u,v}^w(\bs{y}) = 0$ unless there exist $a,b,z$ such that $a\le_L u, b\le_L v, z\ge_L w$ and $C_{a,b}^z\ne 0$. By Lemma~\ref{lem:weak-bruhat-Lambda}, $\lambda_i(a)\le \lambda_i(u), \lambda_i(b)\le \lambda_i(u)$, $\Lambda_i(a)\le \Lambda_i(u), \Lambda_i(b)\le \Lambda_i(u)$, for all $i$, so $\Omega(a,b)\le \Omega(u,v)$ and $\Xi(a,b)\le \Xi(u,v)$. Since $w\le z$, $\BS(w)\le\BS(z)$ and $\FS(w)\le\FS(z)$. Thus,
\[
\BS(w) \le \BS(z) \le \Omega(a,b) \le \Omega(u,v), \qquad \text{and} \qquad \FS(w) \le \FS(z) \le \Xi(a,b) \le \Xi(u,v),
\]
as desired.\qedhere
\end{proof}

\begin{remark}
The formula \eqref{eq:C-uv-w-y-expansion} has multiple drawbacks. First, the positivity property \cite{AGM-positivity} of the $C_{u,v}^w(\bs{y})$ is not clear from \eqref{eq:C-uv-w-y-expansion}, even when the (alternating) positivity of $C_{u,v}^w$ is taken for granted. Second, \eqref{eq:C-uv-w-y-expansion} gives an expansion of $C_{u,v}^w(\bs{y})$ in terms of products of Grothendieck polynomials in $\bs{y}$ and $\ominus\bs{y}$, which is not a basis of $\C(\bs{y})$. The first issue appears very difficult to resolve. The second issue, however, is fixed in the analogous formula for double Schubert polynomials:
\[
c_{u,v}^w(\bs{y})
= \sum_g\left(\sum_{\substack{c,d,z \\ c\le_R u \\ d\le_R v \\ z\ge_L w}} (-1)^{\ell(c) + \ell(d)}c_{c^{-1}u,d^{-1}v}^z c_{c^{-1},d^{-1},w^{-1}z}^g\right) \Sch_g(\bs{y}),
\]
where we take $c_{a,b,c}^z$ to mean $\sum_d c_{a,b}^d c_{c,d}^z$.
\end{remark}

\section{The integer support of the double Grothendieck product}

It follows from Theorem~\ref{thm:double-grothendieck-stability} that the largest simple reflection appearing on the right side of \eqref{eq:double-grothendieck-structure-constants} is $\Xi(u,v)-1$, and the smallest is $\max(1-\Omega(u,v),1)$. As we will see in this section, that is enough to characterize which simple reflections appear.

Given $w\in S_\Z$, let $\supp w$ be the set $\{i_1,\ldots,i_k\}$ for any reduced word $(i_1,\ldots,i_k)$ of $w$. This definition is independent of the reduced word chosen. For a set $P$ of permutations, let $\supp P = \bigcup_{w\in P} \supp w$. Let $\lambda_{i,j}(w) = \sum_{i\le k\le j} \theta_k(w)$ and $\Lambda_{i,j}(w) = \sum_{i\le k\le j} \Theta_k(w)$.

\begin{corollary} \label{cor:integer-support}
If $u,v\in S_\Z$ and $P$ is any of the following sets:
\begin{equation} \label{eq:bs-coeffs}
\begin{aligned}
&\{w\in S_\Z \mid \ol{c_{u,v}^w}\ne 0\},\qquad
&&\{w\in S_\Z \mid \ol{c_{u,v}^w}(\bs{y})\ne 0\},
\\&\{w\in S_\Z \mid \ol{C_{u,v}^w}\ne 0\},\qquad
&&\{w\in S_\Z \mid \ol{C_{u,v}^w}(\bs{y})\ne 0\},
\end{aligned}
\end{equation}
then we have
\begin{equation} \label{eq:bs-support}
\begin{aligned}
\supp P = \{j\in\Z \mid &\exists i\in\supp\{u,v\} \text{ s.t. } \\&\max(1, \lambda_{j,i}(u) + \lambda_{j,i}(v), \Lambda_{i,j}(u) + \Lambda_{i,j}(v)-1) > |j-i|\}.
\end{aligned}
\end{equation}

Furthermore, if $u,v\in S_{\Z_+}$ and $P$ is any of the following sets: 
\begin{equation} \label{eq:ord-coeffs}
\begin{aligned}
&\{w\in S_\Z \mid c_{u,v}^w\ne 0\},\qquad
&&\{w\in S_\Z \mid c_{u,v}^w(\bs{y})\ne 0\},
\\&\{w\in S_\Z \mid C_{u,v}^w\ne 0\},\qquad
&&\{w\in S_\Z \mid C_{u,v}^w(\bs{y})\ne 0\},
\end{aligned}
\end{equation}
then we have 
\begin{equation} \label{eq:non-bs-support}
\begin{aligned}
\supp P = \{j\ge 1 \mid &\exists i\in\supp\{u,v\} \text{ s.t. } \\&\max(1, \lambda_{j,i}(u) + \lambda_{j,i}(v), \Lambda_{i,j}(u) + \Lambda_{i,j}(v)-1) > |j-i|\}.
\end{aligned}
\end{equation}
\end{corollary}

Notice that unless $i=j$, at most one of $\lambda_{j,i}(u) + \lambda_{j,i}(v), \Lambda_{i,j}(u) + \Lambda_{i,j}(v)$ can be nonzero.

\begin{proof}
Let $D$ be the set given by the right side of \eqref{eq:bs-support} or \eqref{eq:non-bs-support}; we show that $D = \supp P$. By Theorem~\ref{thm:double-grothendieck-stability}, $D$ has the correct maximum and minimum elements, accounting for the fact that the largest entry in the integer support is one less than the largest number moved. In addition,
\begin{equation} \label{eq:supp-uv-P}
\supp\{u,v\}\subseteq \supp P
\end{equation}
since every $w\in P$ satisfies $u\le w$, $v\le w$ 
(this follows from Theorem 5.1 and Proposition 4.3 of~\cite{GrahamKumar}; the second half of \cite[Theorem~5.1]{GrahamKumar} ensures that unless $u\le w, v\le w$, all coefficients on the right side of~\cite[(17)]{GrahamKumar} will be zero).
By the first part of \cite[Theorem~1.3]{HardtWallach},
\[
(\max(\supp\{u,v\}),\max D], \subseteq \supp P, \qquad [\min D, \min(\supp\{u,v\}))\subseteq\supp P;
\]
since these values of $j$ can also be seen to be in $D$, the result holds for these $j$.

Thus, we can restrict ourselves to $j\notin\supp\{u,v\}$ such that $\min(\supp\{u,v\}) < j < \max(\supp\{u,v\})$. We will use the following fact. Say that $u,v\in S_\Z$ are \emph{non-overlapping} if $|i-j|\ge 2$ for any $i\in\supp u, j\in\supp v$. When $u$ and $v$ are non-overlapping, Definition~\ref{def:double-groth} can be seen to factor:
\begin{equation} \label{eq:noncontiguous-grothendieck-factorization}
\Groth_{uv}(\bs{x},\bs{y}) = \Groth_{u}(\bs{x},\bs{y})\Groth_v(\bs{x},\bs{y}),
\end{equation}
and in particular, $\supp P = \supp\{u,v\}$ in this case.

Back to general $u$ and $v$, since $j\notin\supp\{u,v\}$, $u$ and $v$ factor as $u=u^{(1)}u^{(2)}, v=v^{(1)}v^{(2)}$, where $u^{(1)},v^{(1)}$ have support contained in $(-\infty,j)$ and $u^{(2)},v^{(2)}$ have support contained in $(j,\infty)$. By \eqref{eq:noncontiguous-grothendieck-factorization},
\begin{equation}
\Groth_u(\bs{x},\bs{y}) = \Groth_{u^{(1)}}(\bs{x},\bs{y})\Groth_{u^{(2)}}(\bs{x},\bs{y}), \qquad\qquad \Groth_v(\bs{x},\bs{y}) = \Groth_{v^{(1)}}(\bs{x},\bs{y})\Groth_{v^{(2)}}(\bs{x},\bs{y}),
\end{equation}
For the sake of concreteness, we now proceed in the case of the back-stable Schubert product; the other cases are similar. If $j\in D$, then either there exists $i<j$ such that $\Lambda_{i,j}(u^{(1)}) + \Lambda_{i,j}(v^{(1)}) > |j-i|$ or there exists $i>j$ such that $\lambda_{j,i}(u^{(2)}) + \lambda_{j,i}(v^{(2)}) > |j-i|$. In the former case, $j\in\supp \{w^{(1)}\in S_\Z \mid \ol{c_{u^{(1)},v^{(1)}}^{w^{(1)}}}\ne 0\}$, and in the latter case, $j\in\supp \{w^{(2)}\in S_\Z \mid\ol{c_{u^{(2)},v^{(2)}}^{w^{(2)}}}\ne 0\}$, so in both cases, by \eqref{eq:supp-uv-P}, $j\in\supp P$.

Finally, if $j\notin D$, then by the reverse of the reasoning in the prior paragraph,
\[
j\notin\supp \{w^{(1)}\in S_\Z \mid\ol{c_{u^{(1)},v^{(1)}}^{w^{(1)}}}\ne 0\}\qquad \text{and}\qquad j\notin\supp \{w^{(2)}\in S_\Z \mid\ol{c_{u^{(2)},v^{(2)}}^{w^{(2)}}}\ne 0\}.
\] Thus, any $w^{(1)}$ and $w^{(2)}$ from the above sets are non-overlapping, and by \eqref{eq:noncontiguous-grothendieck-factorization}, $j\notin\supp P$.\qedhere
\end{proof}

\subsection{Examples}

Let $u,v$ be permutations, and let $P$ be one of the sets in \eqref{eq:bs-support}. We observe the following consequences of Corollary \ref{cor:integer-support}.

\subsubsection{Maximal and minimal support of $P$}

Suppose that $|\supp\{u,v\}|=k$. Then $k\le |\supp P| \le 3k$. The minimum size of $\supp P$ occurs if and only if $\theta_i(u) + \theta_i(v) \le 1$ and $\Theta_i(u) + \Theta_i(v) \le 1$ for all $i$. In particular, if $\supp u$ and $\supp v$ are disjoint, then $\supp P = \supp\{u,v\}$. The maximum size of $\supp P$ occurs for any permutations $u,v\in S_{k+1}$ such that $u(1)=v(1)=k+1$ and $u(k+1)=v(k+1)=1$. Up to shifts and factorization into non-overlapping permutations, all cases where $\supp P$ is maximal have this form.

\subsubsection{Dominant permutations}

A permutation $w\in S_n$ is \emph{dominant} if its Lehmer code is a partition $\lambda$. In this case, we have
\begin{equation} \label{eq:dominant-example}
\theta_i(w) = \begin{cases} 1, & \text{if } 1\le i \le \ell(\lambda), \\ 0, & \text{otherwise,}\end{cases} \qquad
\Theta_i(w) = \begin{cases} 
    1, & \text{if } \exists j<i \text{ s.t. } \lambda_j-\lambda_i \ge i-j,\\
    0, & \text{otherwise.}
\end{cases}
\end{equation}

Suppose $u,v\in S_n$ are dominant permutations, corresponding to partitions $\lambda$ and $\mu$. Then by Corollary~\ref{cor:integer-support}, $\supp P = \{1-\min(\ell(\lambda),\ell(\mu)),\ldots, m\}$ where $m$ is the integer defined by
\[
m = \max_{j\le \max(\ell(\lambda),\ell(\mu))}\left(\lambda_j+\mu_j+j-1\right).
\]
Single Schubert and Grothendieck polynomials for dominant permutations are simply monomials, and the maximum value of $\supp P$ encodes that multiplication. However, the situation is less obvious for the minimum of $\supp P$ and for the case of double Schubert and Grothendieck polynomials.

\begin{figure}[h]
\begin{center}
\scalebox{.6}{$
\begin{array}{c@{\hspace{80pt}}c}
\begin{tikzpicture}
  \path[fill=black] (4,-1) circle (.1);
  \path[fill=black] (5,-2) circle (.1);
  \path[fill=black] (6,-3) circle (.1);
  \path[fill=black] (3,-4) circle (.1);
  \path[fill=black] (7,-5) circle (.1);
  \path[fill=black] (1,-6) circle (.1);
  \path[fill=black] (2,-7) circle (.1);
  \draw (4,-8)--(4,-1)--(8,-1);
  \draw (5,-8)--(5,-2)--(8,-2);
  \draw (6,-8)--(6,-3)--(8,-3);
  \draw (3,-8)--(3,-4)--(8,-4);
  \draw (7,-8)--(7,-5)--(8,-5);
  \draw (1,-8)--(1,-6)--(8,-6);
  \draw (2,-8)--(2,-7)--(8,-7);
  \draw[draw=black,line width=0.5mm] (0.5,-0.5) rectangle ++(1,-1);
  \draw[draw=black,line width=0.5mm] (0.5,-1.5) rectangle ++(1,-1);
  \draw[draw=black,line width=0.5mm] (0.5,-2.5) rectangle ++(1,-1);
  \draw[draw=black,line width=0.5mm] (0.5,-3.5) rectangle ++(1,-1);
  \draw[draw=black,line width=0.5mm] (0.5,-4.5) rectangle ++(1,-1);
  \draw[draw=black,line width=0.5mm] (1.5,-0.5) rectangle ++(1,-1);
  \draw[draw=black,line width=0.5mm] (1.5,-1.5) rectangle ++(1,-1);
  \draw[draw=black,line width=0.5mm] (1.5,-2.5) rectangle ++(1,-1);
  \draw[draw=black,line width=0.5mm] (1.5,-3.5) rectangle ++(1,-1);
  \draw[draw=black,line width=0.5mm] (1.5,-4.5) rectangle ++(1,-1);
  \draw[draw=black,line width=0.5mm] (2.5,-0.5) rectangle ++(1,-1);
  \draw[draw=black,line width=0.5mm] (2.5,-1.5) rectangle ++(1,-1);
  \draw[draw=black,line width=0.5mm] (2.5,-2.5) rectangle ++(1,-1);
  \node at (-0.5,-1) {{\Large $1$}};
  \node at (-0.5,-2) {{\Large $2$}};
  \node at (-0.5,-3) {{\Large $3$}};
  \node at (-0.5,-4) {{\Large $4$}};
  \node at (-0.5,-5) {{\Large $5$}};
  \node at (-0.5,-6) {{\Large $6$}};
  \node at (-0.5,-7) {{\Large $7$}};
  \end{tikzpicture}
  &
  \begin{tikzpicture}
  \path[fill=black] (4,-1) circle (.1);
  \path[fill=black] (5,-2) circle (.1);
  \path[fill=black] (6,-3) circle (.1);
  \path[fill=black] (3,-4) circle (.1);
  \path[fill=black] (7,-5) circle (.1);
  \path[fill=black] (1,-6) circle (.1);
  \path[fill=black] (2,-7) circle (.1);
  \draw (4,0)--(4,-1)--(0,-1);
  \draw (5,0)--(5,-2)--(0,-2);
  \draw (6,0)--(6,-3)--(0,-3);
  \draw (3,0)--(3,-4)--(0,-4);
  \draw (7,0)--(7,-5)--(0,-5);
  \draw (1,0)--(1,-6)--(0,-6);
  \draw (2,0)--(2,-7)--(0,-7);
  \draw[draw=black,line width=0.5mm] (2.5,-5.5) rectangle ++(1,-1);
  \draw[draw=black,line width=0.5mm] (2.5,-6.5) rectangle ++(1,-1);
  \draw[draw=black,line width=0.5mm] (3.5,-3.5) rectangle ++(1,-1);
  \draw[draw=black,line width=0.5mm] (3.5,-5.5) rectangle ++(1,-1);
  \draw[draw=black,line width=0.5mm] (3.5,-6.5) rectangle ++(1,-1);
  \draw[draw=black,line width=0.5mm] (4.5,-3.5) rectangle ++(1,-1);
  \draw[draw=black,line width=0.5mm] (4.5,-5.5) rectangle ++(1,-1);
  \draw[draw=black,line width=0.5mm] (4.5,-6.5) rectangle ++(1,-1);
  \draw[draw=black,line width=0.5mm] (5.5,-3.5) rectangle ++(1,-1);
  \draw[draw=black,line width=0.5mm] (5.5,-5.5) rectangle ++(1,-1);
  \draw[draw=black,line width=0.5mm] (5.5,-6.5) rectangle ++(1,-1);
  \draw[draw=black,line width=0.5mm] (6.5,-5.5) rectangle ++(1,-1);
  \draw[draw=black,line width=0.5mm] (6.5,-6.5) rectangle ++(1,-1);
  \node at (8.5,-1) {{\Large $1$}};
  \node at (8.5,-2) {{\Large $2$}};
  \node at (8.5,-3) {{\Large $3$}};
  \node at (8.5,-4) {{\Large $4$}};
  \node at (8.5,-5) {{\Large $5$}};
  \node at (8.5,-6) {{\Large $6$}};
  \node at (8.5,-7) {{\Large $7$}};
  \end{tikzpicture}
  \end{array}$}
\end{center}
\caption{Rothe and dual Rothe diagrams for the dominant permutation $w = 4563712\in S_7$. The corresponding partition is $\lambda=(3,3,3,2,2)$, and the conjugate partition in $\lambda' = (5,5,3)$. The parts of $\lambda$ appear as the number of boxes in each row of the Rothe diagram, while the parts of $\lambda'$ appear in the rows of the dual Rothe diagram given by \eqref{eq:dominant-example}.}
\label{fig:dominant}
\end{figure}

\subsubsection{Grassmannian permutations}

A permutation $w$ is $k$-\emph{Grassmannian} if it has a single descent at $k$. In this case, the Lehmer code $\code(w) = (c_1(w),\ldots,c_n(w))$ satisfies $c_k(w)\ge c_{k-1}(w) \ge \cdots \ge c_1(w)$, and $c_i(w) = 0$ for $i>k$. Let $\lambda$ be the partition formed from the nonzero parts of $\code(w)$. Similarly, the dual Lehmer code $\dualcode(w) = (d_1(w),\ldots,d_n(w))$ satisfies $d_{k+1}(w)\ge d_{k+2}(w) \ge \cdots \ge d_n(w)$, and $d_i(w) = 0$ for $i\le k$. The nonzero parts of $\dualcode(w)$ form the conjugate partition $\lambda'$. Thus we have
\[
\theta_i(w) = \begin{cases} 1, & \text{if } k-\ell(\lambda) < i \le k, \\ 0, & \text{otherwise,}\end{cases} \qquad \Theta_i(w) = \begin{cases} 1, & \text{if } k < i \le k + \ell(\lambda'), \\ 0, & \text{otherwise,}\end{cases}
\]
as can be seen in Figure~\ref{fig:grassmannian}.

Suppose that $u,v\in S_n$ are $k$-Grassmannian, corresponding to partitions $\lambda$ and $\mu$. (Double) Schubert/Grothendieck products correspond to cup products of (equivariant) cohomology/K-theory classes in the Grassmannian $Gr(n,k)$, so all $w\in P$ are also $k$-Grassmannian. Then by Corollary~\ref{cor:integer-support}, \[
\supp P = \{k-\ell(\lambda)-\ell(\mu)+1,\ldots,k+\ell(\lambda')+\ell(\mu')-1.\}.
\]
Thus, if $w\in P$ corresponds to the partition $\nu$, then we obtain the fact that $\nu$ is contained in a $(\ell(\lambda)+\ell(\mu)) \times (\ell(\lambda')+\ell(\mu'))$ box, recovering a well-known consequence (in the Schubert case) of the Littlewood-Richardson rule.

\begin{figure}[h]
\begin{center}
\scalebox{.6}{$
\begin{array}{c@{\hspace{80pt}}c}
\begin{tikzpicture}
  \path[fill=black] (2,-1) circle (.1);
  \path[fill=black] (3,-2) circle (.1);
  \path[fill=black] (6,-3) circle (.1);
  \path[fill=black] (7,-4) circle (.1);
  \path[fill=black] (1,-5) circle (.1);
  \path[fill=black] (4,-6) circle (.1);
  \path[fill=black] (5,-7) circle (.1);
  \path[fill=black] (8,-8) circle (.1);
  \draw (2,-9)--(2,-1)--(9,-1);
  \draw (3,-9)--(3,-2)--(9,-2);
  \draw (6,-9)--(6,-3)--(9,-3);
  \draw (7,-9)--(7,-4)--(9,-4);
  \draw (1,-9)--(1,-5)--(9,-5);
  \draw (4,-9)--(4,-6)--(9,-6);
  \draw (5,-9)--(5,-7)--(9,-7);
  \draw (8,-9)--(8,-8)--(9,-8);
  \draw[draw=black,line width=0.5mm] (0.5,-0.5) rectangle ++(1,-1);
  \draw[draw=black,line width=0.5mm] (0.5,-1.5) rectangle ++(1,-1);
  \draw[draw=black,line width=0.5mm] (0.5,-2.5) rectangle ++(1,-1);
  \draw[draw=black,line width=0.5mm] (0.5,-3.5) rectangle ++(1,-1);
  \draw[draw=black,line width=0.5mm] (3.5,-2.5) rectangle ++(1,-1);
  \draw[draw=black,line width=0.5mm] (3.5,-3.5) rectangle ++(1,-1);
  \draw[draw=black,line width=0.5mm] (4.5,-2.5) rectangle ++(1,-1);
  \draw[draw=black,line width=0.5mm] (4.5,-3.5) rectangle ++(1,-1);
  \node at (-0.5,-1) {{\Large $1$}};
  \node at (-0.5,-2) {{\Large $2$}};
  \node at (-0.5,-3) {{\Large $3$}};
  \node at (-0.5,-4) {{\Large $4$}};
  \node at (-0.5,-5) {{\Large $5$}};
  \node at (-0.5,-6) {{\Large $6$}};
  \node at (-0.5,-7) {{\Large $7$}};
  \node at (-0.5,-8) {{\Large $8$}};
  \end{tikzpicture}
  &
  \begin{tikzpicture}
  \path[fill=black] (2,-1) circle (.1);
  \path[fill=black] (3,-2) circle (.1);
  \path[fill=black] (6,-3) circle (.1);
  \path[fill=black] (7,-4) circle (.1);
  \path[fill=black] (1,-5) circle (.1);
  \path[fill=black] (4,-6) circle (.1);
  \path[fill=black] (5,-7) circle (.1);
  \path[fill=black] (8,-8) circle (.1);
  \draw (2,0)--(2,-1)--(0,-1);
  \draw (3,0)--(3,-2)--(0,-2);
  \draw (6,0)--(6,-3)--(0,-3);
  \draw (7,0)--(7,-4)--(0,-4);
  \draw (1,0)--(1,-5)--(0,-5);
  \draw (4,0)--(4,-6)--(0,-6);
  \draw (5,0)--(5,-7)--(0,-7);
  \draw (8,0)--(8,-8)--(0,-8);
  \draw[draw=black,line width=0.5mm] (1.5,-4.5) rectangle ++(1,-1);
  \draw[draw=black,line width=0.5mm] (2.5,-4.5) rectangle ++(1,-1);
  \draw[draw=black,line width=0.5mm] (5.5,-4.5) rectangle ++(1,-1);
  \draw[draw=black,line width=0.5mm] (5.5,-5.5) rectangle ++(1,-1);
  \draw[draw=black,line width=0.5mm] (5.5,-6.5) rectangle ++(1,-1);
  \draw[draw=black,line width=0.5mm] (6.5,-4.5) rectangle ++(1,-1);
  \draw[draw=black,line width=0.5mm] (6.5,-5.5) rectangle ++(1,-1);
  \draw[draw=black,line width=0.5mm] (6.5,-6.5) rectangle ++(1,-1);
  \node at (9.5,-1) {{\Large $1$}};
  \node at (9.5,-2) {{\Large $2$}};
  \node at (9.5,-3) {{\Large $3$}};
  \node at (9.5,-4) {{\Large $4$}};
  \node at (9.5,-5) {{\Large $5$}};
  \node at (9.5,-6) {{\Large $6$}};
  \node at (9.5,-7) {{\Large $7$}};
  \node at (9.5,-8) {{\Large $8$}};
  \end{tikzpicture}
  \end{array}$}
\end{center}
\caption{Rothe and dual Rothe diagrams for the $4$-Grassmannian permutation $w = 23671458\in S_8$. The corresponding partition is $\lambda=(3,3,1,1)$, and the conjugate partition is $\lambda'=(4,2,2)$. The parts of $\lambda$ appear as the number of boxes in the first $k$ rows of the Rothe diagram, while the parts of $\lambda'$ appear as the number of boxes in the last $n-k$ rows of the dual Rothe diagram.}
\label{fig:grassmannian}
\end{figure}

\bibliographystyle{siam}
\bibliography{bibliography.bib}

\end{document}